\documentclass[12pt,reqno]{amsproc}
\usepackage[margin=.9in]{geometry}
\usepackage{amsmath,amsfonts,amsthm,amssymb}
\usepackage{nicefrac}
\usepackage{tikz}
\usetikzlibrary{matrix,arrows}

\usepackage[all,2cell,cmtip]{xy}
\UseAllTwocells
\usepackage{epsfig, graphicx, color}
\usepackage{mathpazo} \renewcommand{\mathbf}{\mathbold}

\usepackage{hyperref}

\hypersetup
{
colorlinks,
linkcolor = blue,
citecolor = blue,
urlcolor = blue,
filecolor = blue
}

\usepackage{euscript}

\newcounter{thecounter}
\numberwithin{thecounter}{section}

\newtheorem{proposition}[thecounter]{Proposition}

\newtheorem*{theorem*}{Theorem}
\newtheorem{thm}[thecounter]{Theorem}
\newtheorem{corollary}[thecounter]{Corollary}
\newtheorem{defn}[thecounter]{Definition}
\newtheorem{rem}[thecounter]{Remark}
\newtheorem{ex}[thecounter]{Example}

\def\Z{{\mathbb Z}}

\newcommand{\bA}{\mathbf{A}}
\newcommand{\bB}{\mathbf{B}}
\newcommand{\G}{\mathrm{G}}
\newcommand{\HH}{\mathrm{H}}
\newcommand{\X}{\mathbb{X}}
\newcommand{\Y}{\mathbb{Y}}

\newcommand{\cgog}{\underline{\G}}
\newcommand{\Oh}{\mathbf{O}}

\newcommand{\id}{\mathrm{\text{id}}}
\newcommand{\Quot}{{\X\mkern-1.5mu/\mkern-2.2mu\G}} 

\newcommand{\Grp}{\mathrm{\mbox{\bf Grp}}}
\newcommand{\Sub}{\mathrm{\mbox{\bf Sub}}}
\newcommand{\Fac}{\mathrm{\mbox{\bf Fac}}}
\newcommand{\CG}{\mathrm{\mbox{\bf Com}}}
\newcommand{\EC}{\mathrm{\mbox{\bf Str}}_{\G}}

\renewcommand{\succeq}{\succ}
\newcommand{\tab}{\hspace{.25in}}

\newcommand{\bq}{\mathbf{q}}
\newcommand{\bp}{\mathbf{p}}
\newcommand{\pro}{\bp}

\newcommand{\Ad}{\mathrm{{Ad}}}

\newcommand{\fiber}{\mathbin{\mkern-6mu/\mkern-8.5mu/}\mkern-5mu} 

\newcommand{\bs}{\mathbf{s}}
\newcommand{\com}{\mkern-3mu\circ\mkern-2.4mu}
\newcommand{\cd}{\mkern-2mu\cdot\mkern-1.6mu} 
\newcommand{\inv}{^{-1}}
\newcommand{\inj}{\hookrightarrow}
\newcommand{\surj}{\twoheadrightarrow}
\newcommand{\two}{\Rightarrow}

\newcommand{\cst}{\mathbf{c}}

\newcommand{\btrt}{\blacktriangleright}

\begin{document}

\title{Equivariant Simplicial Reconstruction}
\author{Lisa Carbone, Vidit Nanda and Yusra Naqvi}

\begin{abstract}
We introduce and analyze parallelizable algorithms to compress and accurately reconstruct finite simplicial complexes that have non-trivial automorphisms. The compressed data -- called a complex of groups -- amounts to a functor from (the poset of simplices in) the orbit space to the 2-category of groups, whose higher structure is prescribed by isomorphisms arising from conjugation. Using this functor, we show how to algorithmically recover the original complex up to equivariant simplicial isomorphism. Our algorithms are derived from generalizations (by Bridson-Haefliger, Carbone-Rips and Corson, among others) of the classical Bass-Serre theory for reconstructing group actions on trees.
\end{abstract}

\maketitle

\section*{Introduction}

The ongoing proliferation of large, complicated and vital datasets has sparked considerable activity focused on rendering hitherto-abstract branches of mathematics applicable to processing complex data. Examples of this phenomenon include, but are by no means restricted to, the recent employment of Laplacian eigenfunctions for spectral embedding \cite{belkin:niyogi:03}, of rough paths in machine learning \cite{chevyrev:kormilitzin:16}, of sheaf theory for linear programming \cite{krishnan:14}, of Morse theory for image processing \cite{friedrichs:robins:sheppard:14} and of (co)homology for data analysis \cite{ghrist:09}. In each case, efficient algorithms have catalyzed the percolation of theory to application. And particularly in the last two examples, these algorithms accept as input some cell complex structure (often simplicial or cubical) imposed on the constituent elements of a given dataset. 

Our work here continues these efforts: we adapt a framework originally devised for the study of infinite group actions on non-positively curved spaces \cite{bridson:haefliger:99, serre:80} to the concrete task of efficiently exploiting symmetries to compress and reconstruct finite simplicial complexes. The central contribution of this paper therefore consists of two parallelizable algorithms. The first, called {\tt Compress}, accepts as input a finite simplicial complex $\X$ along with a subgroup $\G$ of its automorphism group, and outputs a compressed structure $\bA$ called a {\em complex of groups} \cite{haefliger:92, corson:92}. This $\bA$ may, for the purposes of these introductory remarks, be regarded as a (typically much smaller) group-weighted simplicial complex. The second algorithm, {\tt Reconstruct}, inverts the first by using the overlaid algebraic data to correctly unfold the complex of groups $\bA$ so that $\X$ is recovered up to $\G$-equivariant isomorphism.

Assume, for the purposes of this introduction, that $\X$ is a triangulated bow-tie:

\begin{figure}[h!]
\includegraphics[scale=0.35]{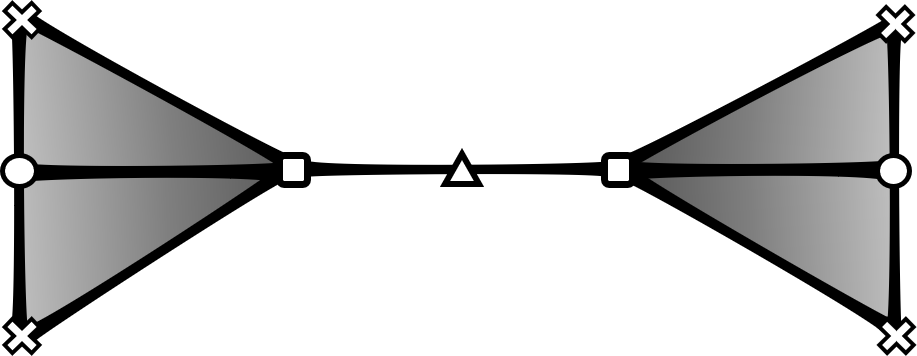}
\end{figure}

\noindent and consider the action of the Klein four-group $\G = \left\langle\sigma,\tau \mid \sigma^2 = \tau^2 = (\sigma\tau)^2=1\right\rangle$ where $\sigma$ and $\tau$ act by reflection across the horizontal and vertical axis through the central vertex $\Delta$ respectively (vertices lying in the same orbit have been decorated similarly). One possible choice of quotient $\Quot$ is the following {\em fundamental domain}, i.e., a subcomplex of $\X$ which intersects each $\G$-orbit exactly once:

\begin{figure}[h!]
\includegraphics[scale=0.4]{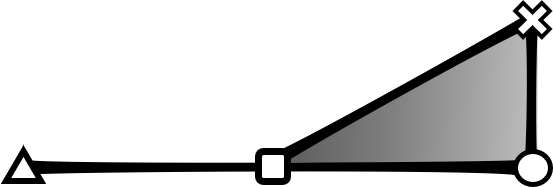}
\end{figure}

It is not too difficult to construct an infinite family of other group actions on different simplicial complexes which produce the same quotient. Clearly, $\Quot$ does not carry sufficient information to recover either $\X$ or the $\G$-action. One therefore seeks the minimal amount of additional data necessary for such a reconstruction. Not surprisingly, the extra machinery required varies (in complexity, if not nature) depending on whether $\X$ is a tree \cite{serre:80}, a graph \cite{bass:93, dicks:dunwoody:89}, a smooth manifold \cite{grove:ziller:12}, a Coxeter complex \cite{davis:08}, or the classifying space of a small loopfree category \cite{haefliger:92}. 

In any event, the complex of groups $\bA$ for the action described above assigns to each simplex $y$ of $\Quot$ a stabilizer subgroup $\bA_y \leq \G$ which fixes some simplex of $\X$ lying in the corresponding orbit class:

\begin{figure}[h!]
\includegraphics[scale=0.4]{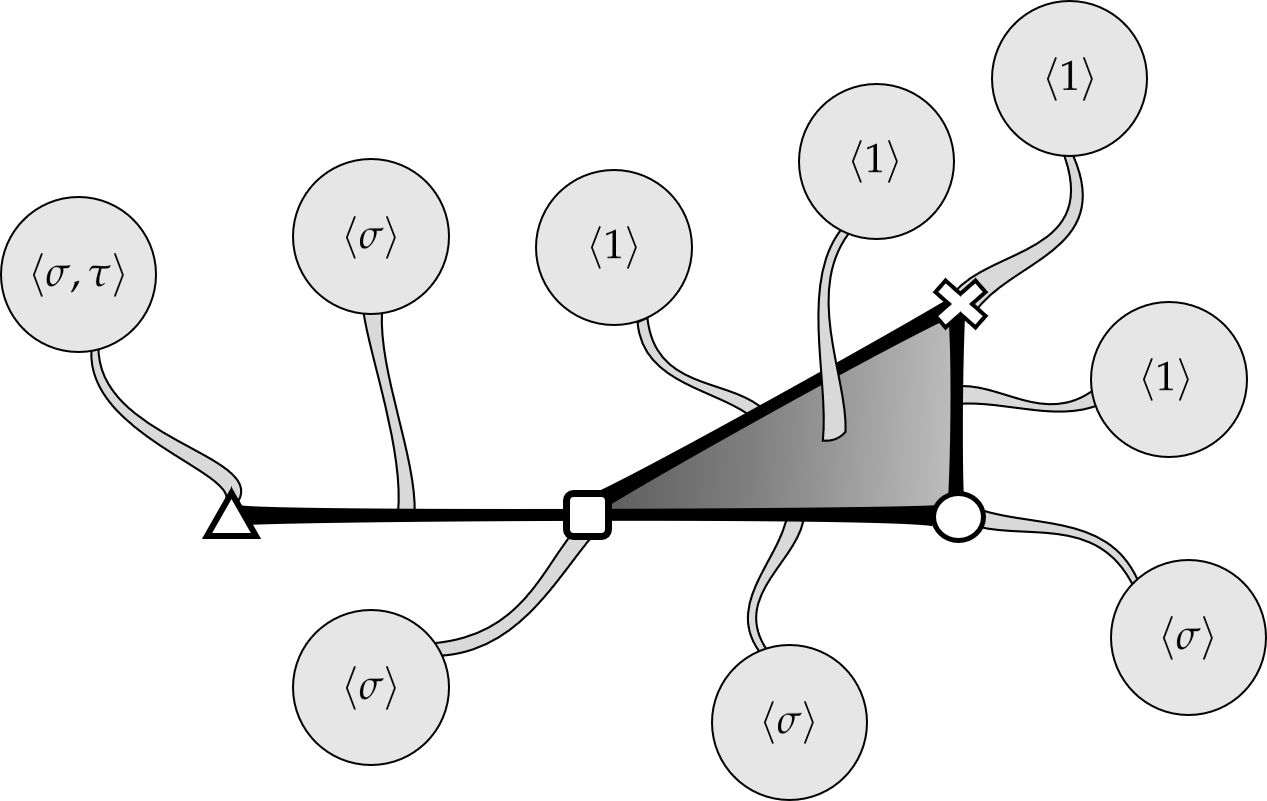}
\end{figure}

\noindent The group assigned to each simplex includes into the groups assigned to simplices in its boundary, so our $\bA$ is a group-valued {\em cellular cosheaf} \cite{curry:13}, i.e., a functor from the poset of simplices in $\Quot$ to the lattice of subgroups of $\G$. It should be noted that -- at least in our simple example -- $\G$ is the colimit of $\bA$. An appeal to the orbit-stabilizer theorem also guarantees that the number of simplices in $\X$ whose orbit class is represented by a simplex $y$ in $\Quot$ equals the usual index $[\G:\bA_y]$ which counts cosets of $\bA_y$ in $\G$. Our next task, therefore, is to determine the correct face relations among these simplices.

At this stage, the scope of the difficulty starts to become apparent: how should one glue the coset-indexed simplices below so that the result is (isomorphic to) $\X$?

\begin{figure}[h!]
\includegraphics[scale=0.4]{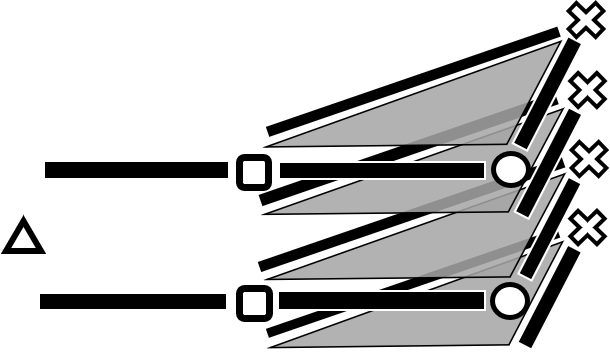}
\end{figure}

\noindent Knowledge of the subgroup indices $[\G:\bA_y]$ across all simplices $y$ in $\Quot$ does not uniquely specify the desired face relations between our coset-simplices. Making unfortunate choices of attachments, even between vertices and edges, could easily backfire: 

\begin{figure}[h!]
\includegraphics[scale=0.4]{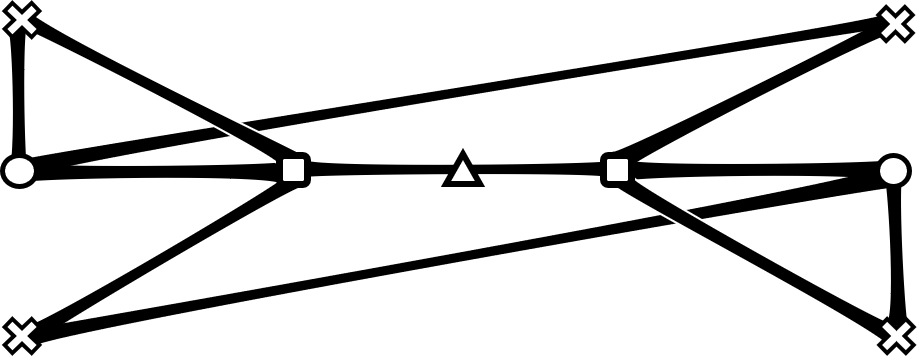}
\end{figure}

Here is a partial summary of the challenges which must be overcome before one can reconstruct more general group actions on more general complexes than our $\G$ and $\X$:
\begin{enumerate}
\item The quotient $\Quot$ may not be a simplicial complex.
\item Even if $\Quot$ is simplicial, it may not be a subcomplex of $\X$.
\item If $\X$ is not simply connected, then $\G$ may not be the colimit of $\bA$.
\item The groups $\bA_y$ are only determined up to conjugation in $\G$.
\item It is unclear how one should glue coset-simplices to recover $\X$.
\end{enumerate}

Obstacle (1) above is bypassed by imposing a mild regularity constraint on the $\G$-action, which is always satisfied after barycentric subdivision. In order to address the remaining challenges, one must carefully keep track of how the various $\bA_y$ embed within each other inside $\G$. Perhaps the most streamlined way to accomplish this is to view $\bA$ as a {\em pseudofunctor}\footnote{Roughly, this means that $\bA$ is associative only up to inner automorphisms. See Def \ref{def:cogs} for details.} by enhancing its target into a 2-category as follows. Given two homomorphisms $\phi$ and $\phi'$ between the same pair of groups, the set of 2-morphisms $\phi \two \phi'$ is given by all elements $g$ in the codomain (if any) which satisfy $\phi(\bullet) = g \cd \phi'(\bullet) \cd g\inv$. The remarkable advantage of this pseudofunctorial perspective, which we exploit in Sec \ref{sec:recon}, is that knowledge of the 2-morphisms lying in the image of $\bA$ solves {\em all} the problems (2)-(5).

The following result motivates and underlies much of our work. In its statement, $\cgog^\Y$ denotes the constant complex of groups over $\Y$, which assigns the group $\G$ to all simplices and identity maps to all face relations in sight.
\begin{theorem*}[Thm 2.18 and Cor 2.19 of Sec III.C in \cite{bridson:haefliger:99}]
For each finite simplicial complex $\Y$ and finite group $\G$, there is an equivalence 
\[
\left[\begin{matrix}
\text{Regular $\G$-actions on finite} \\
\text{simplicial complexes $\X$ with} \\
\text{ quotient }\Quot = \Y
\end{matrix}\right] \stackrel{\simeq}{\longleftrightarrow}
\left[\begin{matrix}
\text{Certain equivalence classes} \\
\text{of complexes of groups over $\Y$} \\
\text{that map injectively into $\cgog^\Y$}
\end{matrix}\right]
\]
\end{theorem*}

\subsection*{Related Work and Outline} Much of the theoretical machinery presented here (including the theorem above) will be quite familiar, or at least unsurprising, to those with expertise in certain areas of geometric group theory and equivariant algebraic topology. To the best of our knowledge, however, none of the existing literature on reconstructing group actions seriously confronts its algorithmic aspects. To accomplish that task here, we have adapted the treatment in the comprehensive text of Bridson and Haefliger \cite[Ch III.C]{bridson:haefliger:99}, a recent framework developed by the first author with Rips \cite{carbone:rips:13}, and the pseudofunctorial viewpoint mentioned above (which was also employed in Fiore-L\"uck-Sauer \cite[Sec 8]{fiore:luck:sauer11} for a different purpose). In an effort to avoid impeding the progress of readers unfamiliar with this body of work, we have also provided a summary of the basic constructions and fundamental results which are required to verify the correctness of our algorithms.

Our hope is that these algorithms will eventually be used for compressing simplicial complexes built around data points sampled from symmetric manifolds. It is, however, an unfortunate by-product of almost all discretization techniques that symmetries of smooth objects are not directly inherited by their finite approximations. For instance, a dense point cloud sampled uniformly at random from the unit sphere (viewed as a submanifold of Euclidean space) will not be fixed by any non-trivial element of the orthogonal group. Any reasonable framework for inferring automorphisms of symmetric manifolds from random samples will, in all probability, require methods to efficiently discover, quantify and analyze {\em approximate} symmetries of finite data. While that grail-quest lies far beyond the scope of our work here, we direct interested parties to promising geometric \cite{mitra:guibas:pauly:06} and statistical \cite[Sec 3]{gao:brodzki:mukherjee:16} efforts in its general direction. 

The rest of this paper is organized as follows. Sec \ref{sec:back} contains preliminary material regarding complexes of groups, Sec \ref{sec:comp} and Sec \ref{sec:recon} describe the voyage from a simplicial group action to the associated complex of groups and back, while Sec \ref{sec:algos} contains the two promised algorithms {\tt Compress} and {\tt Reconstruct} along with their detailed complexity estimates, which are recorded in Thm \ref{thm:comp} and Thm \ref{thm:recon} respectively.

\section{Backgound} \label{sec:back}

For combinatorial and homological perspectives on simplicial group actions, we invite the reader to consult \cite{babson:kozlov:05} and \cite[Ch III.1]{bredon:72} respectively. General introductions to simplicial complexes, group actions and 2-categories may be found in \cite[Ch III]{spanier:66}, \cite[Ch II.4]{hungerford:74} and \cite{lack:10} respectively. 
 
\subsection{Simplicial Group Actions and Quotients} \label{ssec:simpact}

Fix a finite simplicial complex $\X$ and a finite group $\G$ which acts on $\X$ via simplicial automorphisms. For each simplex $x$ in $\X$, we have an {\em orbit} $\G x = \{g \cd x \mid g \in \G\}$, which is a subset of $\X$, and a {\em stabilizer} $\G_x = \{g \in \G \mid g \cd x = x\}$, which is a subgroup of $\G$. The following definition is adapted from \cite[Ch III.1]{bredon:72}.

\begin{defn}
\label{def:reg}
{\em
The action of $\G$ on $\X$ is called {\bf regular} if the following two properties hold for every simplex $x$ of $\X$. Letting $(v_0,\ldots,v_d)$ denote the vertices of $x$,  
\begin{enumerate}
\item every $g$ in $\G_x$ must satisfy $g \cd v_i = v_i$ for all $i$ in $\{0,\ldots,d\}$, and
\item given $g_0, \ldots, g_d \in \G$, if $x' = (g_0\cd v_0, \ldots, g_d \cd v_d) \in \X$ then $x'$ must lie in $\G x$.
\end{enumerate}
}
\end{defn}
Regular actions are actions {\em without inversion} \cite{bass:93, serre:80} in the special case when $\X$ is a graph (i.e., $\dim \X = 1$), since the first requirement of the definition above prohibits the $\G$-action from interchanging the two vertices of a given edge. We will assume henceforth that $\G$ acts regularly\footnote{Fortunately, regularity is not a severe requirement on group actions --- any automorphism can be made regular via passage to the second barycentric subdivision of $\X$ (see condition (B) in \cite[Ch III.1]{bredon:72}).} in order to avail ourselves of three pleasant consequences. First, if $y$ is a face of $x$ (written $x \succeq y$), then the corresponding stabilizers satisfy the subgroup relation $\G_x \leq \G_y$ because any $g$ which fixes $x$ is forced to fix all the vertices of $x$ (and hence, $y$) pointwise. Second, the orbits assemble to form a quotient simplicial complex so that the natural projection from $\X$ is a simplicial map. And third, $\G$ acts transitively on the fibers of this map.  
\begin{defn}
{\em
The {\bf orbit space} or {\bf quotient} $\Quot$ associated to the action of $\G$ on $\X$ is the simplicial complex defined as follows. Its vertices are the $\G$-orbits of vertices in $\X$, and a $d$-simplex in $\Quot$ is spanned by $(d+1)$ distinct orbits $(\G v_0,\ldots, \G v_d)$ if and only if there exists some simplex $(u_0,\ldots,u_d)$ in $\X$ with $u_i \in \G v_i$ for each $i$.   
}
\end{defn}
By construction, there is a canonical surjective simplicial map $\pro_\G:\X \surj \Quot$, called the {\bf orbit map}, which sends each simplex of $\X$ to its $\G$-orbit in $\Quot$. And by the second regularity requirement, if two simplices $x$ and $x'$ of $\X$ satisfy $\pro_G(x) = \pro_G(x')$ then some $g \in \G$ satisfies $g \cd x = x'$.

\begin{defn}
{\em
A {\em fundamental domain} for the action of $\G$ on $\X$ is a subcomplex $\X' \subset \X$ which contains exactly one simplex from each orbit.
}
\end{defn}
Whenever a fundamental domain exists for a given action (as with the bow-tie from the Introduction), the quotient space $\X/\G$ may be viewed as a subcomplex of $\X$. However, not every action admits a fundamental domain.

\begin{ex} 
\label{ex:rot2splx}
{\em The cyclic group (on three elements) $\G = C_3$ acts by rotation on the subdivided standard $2$-simplex $\X$ shown below:

\begin{figure}[h!]
\includegraphics[scale=0.17]{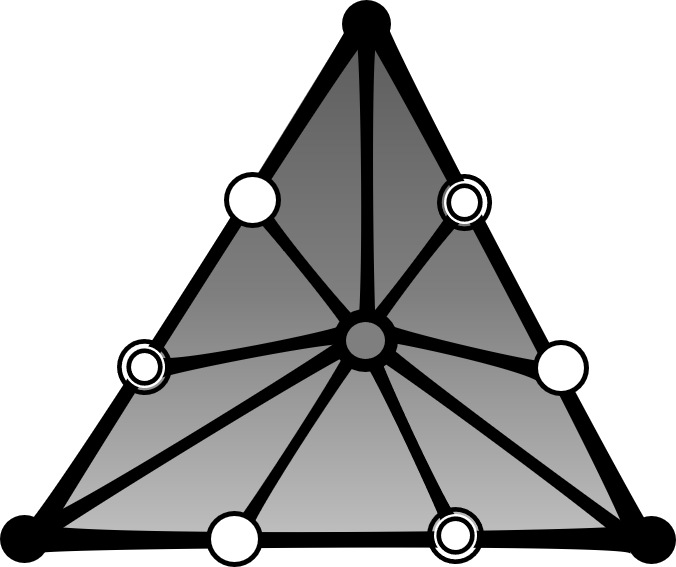}
\end{figure}

\noindent Since this action is regular, $\Quot$ inherits a natural simplicial complex structure (given by a coarser subdivision of the $2$-simplex):

\begin{figure}[h!]
\includegraphics[scale=0.2]{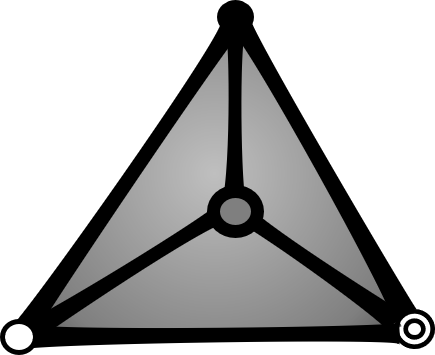}
\end{figure}

\noindent One can further illustrate $\pro_G$ as a simplicial map from $\X$ down to $\Quot$, so that the fiber $\pro_G\inv$ over a simplex in $\Quot$ is the collection of all simplices from $\X$ in the associated orbit:

\begin{figure}[h!]
\includegraphics[scale=0.3]{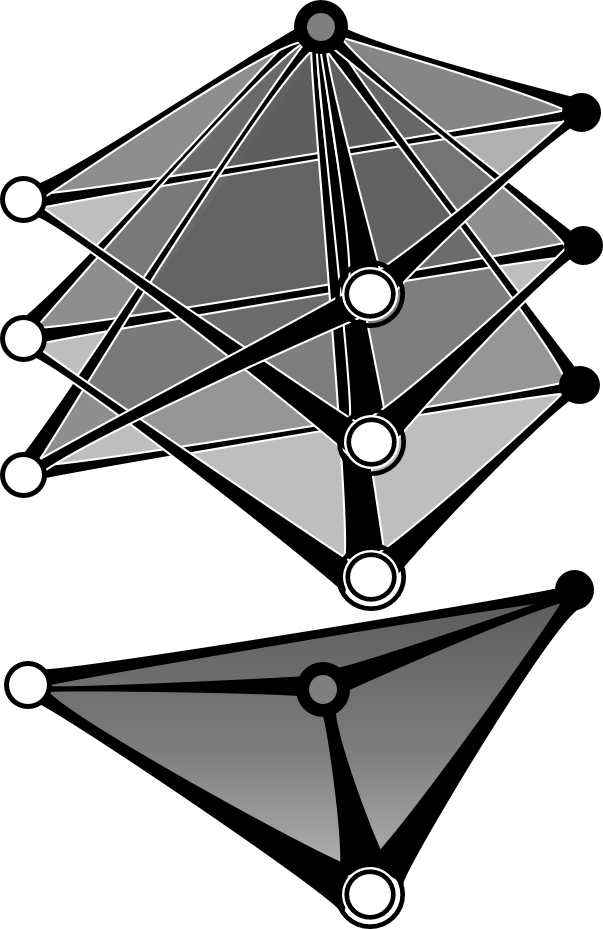}
\end{figure}

\noindent On the other hand, $\G$ also acts by rotation on the quotient $\Y = \Quot$, but this action fails to satisfy the second requirement of regularity. As a consequence, the successive quotient of $\Y$ by $\G$ is no longer a simplicial complex.
}
\end{ex}

\subsection{The 2-Category of Groups}

Consider a pair of groups $\G_0, \G_1$ and group homomorphisms $\phi, \phi': \G_0 \to \G_1$. We write $g: \phi \two \phi'$ if some group element $g \in \G_1$ relates $\phi$ and $\phi'$ by conjugation in the sense that 
\[
\phi = \Ad(g) \com \phi', \text{ i.e., } \phi(h) = g \cd \phi'(h) \cd g\inv \text{ for all }h \in \G_0. 
\] The homomorphisms $\G_0 \to \G_1$ thus form the objects of a groupoid\footnote{By {\em groupoid} we mean a small category all of whose morphisms are invertible.}, which we will denote by $\Grp(\G_0,\G_1)$ --- its morphisms are given by such $g:\phi \two \phi'$ and the composition law is inherited from multiplication in $\G_1$ as follows. Given
\[
\phi \stackrel{g}{\two} \phi' \stackrel{g'}{\two} \phi''
\]
where $\phi''$ is yet another homomorphism $\G_0 \to \G_1$, a straightforward calculation reveals that the product $g \cd g'$ in $\G_1$ satisfies $g\cd g': \phi \two \phi''$, and hence that $g\inv:\phi' \two \phi$ serves as the inverse to $g$. In order to formally distinguish a product in the groupoid $\Grp(\G_0,\G_1)$ from the corresponding (contravariant) product in the group $\G_1$, we will denote the former by $g' * g$ and call it the {\em vertical} composite of $g$ with $g'$. 

Given $g: \phi \two \phi'$ in $\Grp(\G_0, \G_1)$ and $h: \psi \two \psi'$ in $\Grp(\G_1,\G_2)$, define their {\em horizontal} composite $h \com g$ as the group element $\psi(g) \cd h$ in $\G_2$. Two straightforward calculations reveal the following:
\begin{enumerate}
	\item across any triple of groups $\G_0, \G_1$ and $\G_2$, horizontal composition yields a map of groupoids
	\[
	\com: \Grp(\G_1,\G_2) \times \Grp(\G_0,\G_1) \to \Grp(\G_0,\G_2),
	\]
	which extends the ordinary composition for group homomorphisms, and
	\item the horizontal and vertical compositions satisfy the {\bf interchange law}, meaning that across any diagram of the form
	\[
	\xymatrixcolsep{1in}
	\xymatrix
	{
		\G_0 \ar@/^2.5em/[r]|{\phi}_*+<1.1em>{\Downarrow g} \ar@{->}[r]|{\phi'}_*+<1em>{\Downarrow g'} \ar@/_2.5em/[r]|{\phi''}  & \G_1 \ar@/^2.5em/[r]|{\psi}_*+<1.1em>{\Downarrow h} \ar@{->}[r]|{\psi'}_*+<1em>{\Downarrow h'} \ar@/_2.5em/[r]|{\psi''} & \G_2,
	} 
	\]
	we have an equality
	\[
	(h' \com g') * (h \com g) = (h' * h) \com (g' * g)
	\]
	in the groupoid $\Grp(\G_0,\G_2)$.
\end{enumerate}
These two facts allow us to define a higher category of groups \cite[Ex 1.3]{martins-ferreira:06}, \cite[Def 8.1]{fiore:luck:sauer11}.
\begin{defn}
\label{def:2grp}
{\em
The {\bf 2-category of groups}, denoted $\Grp$, is defined as follows:
\begin{enumerate} 
\item its objects are all groups, 
\item the 1-morphisms $\phi:\G \to \G'$ are the usual group homomorphisms, and
\item the 2-morphisms $g: \phi \two \phi'$ are given by conjugation, i.e., $\phi = \Ad(g)\com\phi'$.
\end{enumerate}  
The 1-morphisms are endowed with the usual composition $\circ$ for group homomorphisms, while the horizontal and vertical composition of 2-morphisms is given by $\circ$ and $*$ respectively.
}
\end{defn}
In modern parlance, $\Grp$ is a {\em strict (2,1)-category}, which is to say that all compositions are associative, all identity elements are unique (rather than being defined only up to coherent isomorphisms), and all 2-morphisms are invertible. Dispensing with these subtleties in the interest of brevity, we will simply call $\Grp$ a 2-category henceforth.

\subsection{Complexes of Groups}

Variants of the following definition appeared simultaneously in the work of Haefliger \cite{haefliger:92} and Corson \cite{corson:92}.

\begin{defn}
\label{def:cogs}
{\em
A {\bf complex of groups} $\bA$ over a simplicial complex $\Y$  assigns to each
\begin{enumerate}
\item simplex $x$, a {\em finite} group $\bA_x$,
\item pair $x \succeq y$, an {\em injective} homomorphism $\bA_{x \succeq y}:\bA_x \inj \bA_{y}$, and
\item triple $x \succeq y \succeq z$, a 2-morphism $\bA_{x \succeq y \succeq z}:\bA_{y \succeq z}\com\bA_{x \succeq y} \two \bA_{x \succeq z}$,
\end{enumerate}
subject to the following constraints:
\begin{enumerate}
\item $\bA_{x\succeq x}$ is the identity  $\bA_x \to \bA_x$, 
\item $\bA_{x \succeq x \succeq y}$ and $\bA_{x \succeq y \succeq y}$ are identities $\bA_{x \succeq y} \two \bA_{x \succeq y}$, and
\item \label{eq:cocycle} for any simplices $w \succeq x \succeq y \succeq z$ in $\Y$, the following relation holds in the group $\bA_z$:
\begin{align*}
\bA_{y \succeq z}(\bA_{w \succeq x \succeq y}) \cd \bA_{w \succeq y \succeq z} = \bA_{x \succeq y \succeq z} \cd  \bA_{w \succeq x \succeq z}.
\end{align*}
\end{enumerate}
The last constraint is called the {\em cocycle condition} \cite[Ch III.C.2]{bridson:haefliger:99}.
}
\end{defn} 

In other words, a complex of groups is a pseudofunctor $\bA:\Fac(\Y) \to \Grp$ from the poset $\Fac(\Y)$ of simplices in $\Y$ (ordered by the is-a-face-of relation) to the 2-category of groups from Definition \ref{def:2grp}. This means that the homomorphism $\bA_{x \succeq z}$ only equals the expected composite $\bA_{y \succeq z} \com \bA_{x \succeq y}$ up to conjugation as prescribed by the 2-morphism $\bA_{x \succeq y \succeq z}$ (these 2-morphisms are called {\em twisting elements} in \cite{bridson:haefliger:99}). Similarly, $\bA$ need not send identity 2-morphisms in its domain to identities in its codomain. The cocycle condition, however, ensures associativity in the assignment of 2-morphisms by insisting that the following diagram commutes in the groupoid $\Grp(\bA_w,\bA_z)$:
\[
\xymatrixrowsep{.5in}
\xymatrixcolsep{.75in}
\xymatrix{
	\bA_{y \succeq z} \com \bA_{x \succeq y} \com \bA_{w \succeq x} \ar@{=>}[d]_-{\bA_{x\succeq y \succeq z}\circ \id~} \ar@{=>}[r]^-{\id \circ \bA_{w\succeq x \succeq y}} & \bA_{y \succeq z} \com \bA_{w \succeq y} \ar@{=>}[d]^-{~\bA_{w \succeq y \succeq z}} \\
	\bA_{x \succeq z} \com \bA_{w \succeq x} \ar@{=>}[r]_-{\bA_{w \succeq x \succeq z}}& \bA_{w \succeq z}
}
\]

We have restricted our attention here to complexes of {\em finite} groups for algorithmic reasons; the preceding definition and subsequent ones (can be made to) extend almost verbatim to bifunctors from the poset of open sets in a reasonable topological space to the 2-category of arbitrary groups. The finite complexes of groups defined above may thus be viewed as concrete, combinatorial incarnations of orbifolds \cite{moerdijk:pronk:97, moerdijk:pronk:99} and  topological Deligne-Mumford stacks \cite[Sec 19.5]{noohi:05}.

\begin{ex} 
{\em Given a finite group $\G$ (with identity element $1_\G$) and simplicial complex $\Y$, the {\bf constant} $\G$-valued complex of groups over $\Y$ is denoted $\cgog^\Y:\Fac(\Y) \to \Grp$, and defined as follows: 
\begin{enumerate}
\item each simplex $y$ is assigned the same group $\cgog^\Y_y = \G$, 
\item each pair $x \succeq y$ is assigned the identity homomorphism $\cgog^\Y_{x \succeq y} = \id: \G \to \G$, and 
\item each triple $x \succeq y \succeq z$ is assigned the identity $\cgog^\Y_{x \succeq y \succeq z} = 1_\G: \id \two \id$.
\end{enumerate}
}
\end{ex}

In order to compare complexes of groups, we require a convenient notion of morphisms between them (\cite[Def 2.5]{haefliger:92} and \cite[Def 12]{lim:thomas:08}).
\begin{defn}
{\em
A {\bf morphism} $\Phi:\bA \to \bB$ of complexes of groups over (the same simplicial complex) $\Y$ assigns
\begin{enumerate}
\item to each simplex $y$ a 1-morphism $\Phi_y:\bA_y \to \bB_y$ in $\Grp$, and
\item to each pair $x \succeq y$ a 2-morphism $\Phi_{x \succ y}: \Phi_y \com \bA_{x \succeq y} \two \bB_{x \succeq y}\com\Phi_x$ in $\Grp$:
\[
\xymatrixcolsep{0.7in}
\xymatrixrowsep{0.5in}
\xymatrix{
\bA_x \ar@{->}[d]_{\bA_{x\succeq y}}  \ar@{->}[r]^{\Phi_x}  & \bB_x   \ar@{->}[d]^{\bB_{x\succeq y}} \\
\bA_y \ar@{->}[r]_{\Phi_y} \ar@{=>}[ur]|-{\Phi_{x \succeq y}} & \bB_y
}
\]
\end{enumerate}
so that the following two axioms hold:
\begin{enumerate}
	\item the {\em identity} axiom requires $\Phi_{y \succeq y}$ to be the identity $\Phi_y \two \Phi_y$ for each simplex $y$, whereas
	\item \label{eq:coh} the {\em coherence} axiom imposes a relation
	\[
	\Phi_z(\bA_{x \succeq y \succeq z}) \cd \Phi_{x \succeq z} = \Phi_{y \succeq z} \cd \bB_{y \succeq z} (\Phi_{x \succeq y}) \cd  \bB_{x \succeq y \succeq z}.
	\]
	in the group $\bB_z$.
\end{enumerate}
}
\end{defn}

Thus, $\Phi$ is a pseudonatural transformation \cite[Sec 1.2]{leinster:98} between pseudofunctors $\bA,\bB:\Fac(\Y) \to \Grp$. The coherence axiom requires the commutativity of a certain pentagon in the groupoid $\Grp(\bA_x,\bB_z)$ for any three simplices $x \succeq y \succeq z$. We indicate its vertices below:
\[
\xymatrixcolsep{-.32in}
\xymatrixrowsep{0.25in}
\xymatrix{
 & \Phi_z \com \bA_{y \succeq z}\com \bA_{x \succeq y} \ar@{=>}[rr] \ar@{=>}[dl] & & \bB_{y \succeq z} \com \Phi_y \com \bA_{x \succeq y} \ar@{=>}[dr]  & \\
\Phi_z \com \bA_{x \succeq z} \ar@{=>}[drr] & & & &  \bB_{y \succeq z} \com \bB_{x \succeq y} \com \Phi_x \ar@{=>}[dll] \\
& & \bB_{x \succeq z} \com \Phi_x & &
}
\]
and encourage the reader to accurately decorate its edges. We call $\Phi:\bA \to \bB$ {\bf injective} if each $\Phi_y:\bA_y \to \bB_y$ is an injective group homomorphism. And if each $\Phi_y$ is an isomorphism between $\bA_y$ and $\bB_y$, then one can easily construct an inverse pseudofunctor $\bB \to \bA$, so in this case $\Phi$ itself is an {\bf isomorphism} between $\bA$ and $\bB$. 

\begin{rem}
\label{rem:injconst}
{\em
For our purposes here, the most important morphisms of complexes of groups over a simplicial complex $\Y$ will be the {\em injective ones from an arbitrary domain to a constant codomain}, i.e., $\Phi:\bA \to \cgog^\Y$ for some fixed group $\G$. In this special case, $\Phi$ assigns to each
\begin{enumerate} 
\item simplex $y$ an injective group homomorphism $\Phi_y:\bA_y \inj \G$, and 
\item pair $x \succeq y$ a 2-morphism $\Phi_{x \succeq y}:\Phi_y \com \bA_{x \succeq y} \two \Phi_x$ in $\Grp$\footnote{As in Def \ref{def:2grp},  $\Phi_{x \succeq y}$ is an element of $\G$ satisfying $\Ad(\Phi_{x \succeq y})\Phi_{x} = \Phi_y \com \bA_{x \succeq y}$.},
\end{enumerate} 
so that every $\Phi_{y \succ y}$ is the identity, and the following relation holds in $\G$ across all triples $x \succeq y \succeq z$:
\begin{align}
\label{eq:injcoh}
\Phi_z(\bA_{x \succeq y \succeq z}) \cd \Phi_{x \succeq z}  = \Phi_{y \succeq z} \cd  \Phi_{x \succeq y}.
\end{align}
The existence of such a $\Phi$ implies that all the $\bA_y$ are simultaneously and coherently realizable as subgroups of $\G$. In other words, $\G$ forms a cocone \cite[Ch III.4]{maclane:71} for the diagram in $\Grp$ parametrized by the functor $\bA$ over the poset $\Fac(\Y)$.
}
\end{rem}

In the sequel, we will compare not only complexes of groups, but also their morphisms. We achieve the latter by using the 2-categorical structure of $\Grp$ locally over each simplex $y$ in $\Y$ as follows. 
\begin{defn}
{\em
Given two morphisms $\Phi,\Psi:\bA \to \bB$ of complexes of groups  over $\Y$, a {\bf homotopy} $\theta: \Phi \two \Psi$ is a (contravariant) collection of 2-morphisms $\{\theta_y:\Psi_y \two \Phi_y\}$ in $\Grp$, indexed by simplices $y$ of $\Y$, so that for each face relation $x \succeq y$ the following square commutes:
\[
\xymatrixcolsep{0.7in}
\xymatrixrowsep{0.4in}
\xymatrix{
\Psi_y \com \bA_{x \succeq y} \ar@{=>}[r]^{\Psi_{x\succeq y}}  \ar@{=>}[d]_{\theta_y \circ \id} & \bB_{x \succeq y} \com \Psi_{x} \ar@{=>}[d]^{\id \circ \theta_{x}} \\
\Phi_y \com \bA_{x \succeq y} \ar@{=>}[r]_{\Phi_{x\succeq y}} & \bB_{x \succeq y} \com \Phi_{x} 
}
\]
In other words, we have a relation
\[
 \Psi_{x \succeq y} \cd  \bB_{x \succeq y}(\theta_x) =  \theta_{y} \cd \Phi_{x \succeq y}
\]
in the group $\bB_{y}$. 
}
\end{defn}
Horizontal and vertical compositions for homotopies are defined simplex-wise. Since such homotopies relate pseudonatural transformations between pseudofunctors, they form examples of {\em modifications} \cite[Sec 1.3]{leinster:98}. Complexes of groups over $\Y$, injective morphisms between them, and homotopies between those morphisms also form a higher category.
\begin{defn}
\label{def:2catcog}
{\em
The {\bf 2-category of complexes of groups over $\Y$}, written $\CG(\Y)$, is described by the following data:
\begin{enumerate}
\item its objects are all complexes of groups $\bA: \Fac(\Y) \to \Grp$ over $\Y$,
\item its 1-morphisms are all {\em injective} morphisms $\Phi:\bA \to \bB$, and
\item its 2-morphisms are all homotopies $\theta:\Phi \two \Psi$.
\end{enumerate}
}
\end{defn}

\section{Compression} \label{sec:comp}

Throughout this section, $\X$ is a finite, connected simplicial complex while $\G$ is a fixed subgroup of its regular automorphisms. Let $\Y$ denote the  orbit space $\Quot$ and $\pro_G:\X \to \Y$ the orbit map. Here we construct a pair $(\bA,\Phi)$, where
\begin{enumerate}
\item  $\bA:\Fac(\Y) \to \Grp$ is a complex of groups over $\Y$, and
\item $\Phi:\bA \to \cgog^\Y$ is an injective morphism to the constant complex of groups.
\end{enumerate}
Following \cite[Sec III.C.2.9]{bridson:haefliger:99}, we explicitly describe $\bA$ and $\Phi$ by making certain ad-hoc local choices. Fortunately, it turns out that different choices lead to isomorphic constructions.

\subsection{Lifts and Transfers}\label{ssec:liftrans} 

The surjectivity of $\pro_\G$ guarantees that each simplex $y$ of $\Y$ admits a {\em lift} in the sense that some simplex $x$ in $\X$ satisfies $\pro_\G(x) = y$. We select such lifts arbitrarily for all such $y$, emphasizing that face relations of the form $y \succeq y'$ in $\Y$ may not ascend to relations of the form $x \succeq x'$ in $\X$ among lifts. As a result, the collection of lifts (which is sometimes called a {\em $\G$-transversal \cite[Sec I.1.3]{dicks:dunwoody:89}}) might be quite far from forming a subcomplex of $\X$. The black simplices below form a complete and valid choice of lifts for the action from Example \ref{ex:rot2splx}, but they manifestly do not constitute a simplicial subcomplex:

\begin{figure}[h!]
\includegraphics[scale=0.35]{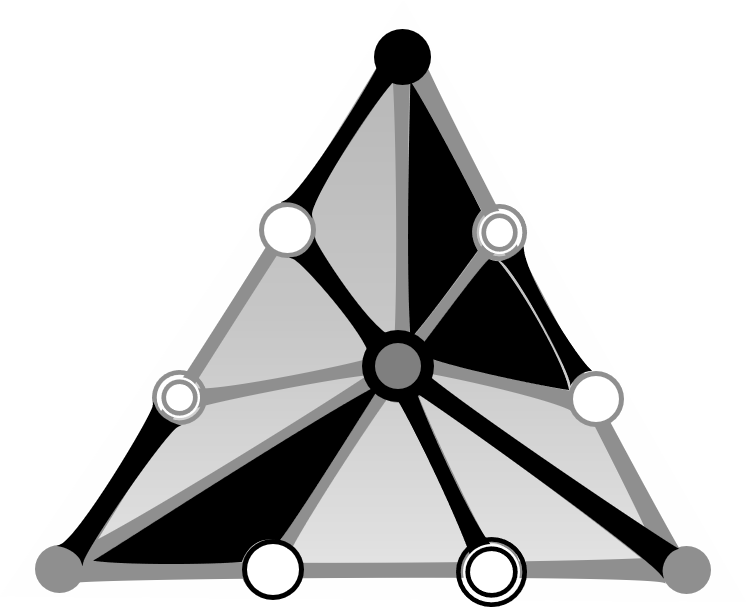}
\end{figure}

In order to address the defect illustrated above, one seeks to relate the lifts $x$ and $x'$ associated to pairs of distinct adjacent simplices $y \succeq y'$ in $\Y$. Let $z$ be the unique face of $x$ in $\X$ which satisfies $\pro_\G(z) = y'$, i.e., $z$ is the face whose vertices lie in orbits determined by $y'$. Since $\G$ acts regularly on $\X$, there exists an element $g \in \G$, not necessarily unique, so that $g \cd z = x'$. Thus, the best that one can expect in general is $x \succeq g\inv x'$ in $\X$ for some $g$ whenever $y \succeq y'$ holds in $\Y$. We arbitrarily select one such $g = g_{y \succeq y'}$, and call it the {\em transfer}\footnote{These transfers give rise to the monodromy elements of \cite{carbone:rips:13}.} associated to the relation $y \succeq y'$. On the other hand, if the lifts did happen to satisfy $x \succeq x'$ in $\X$, then one could simplify (computational) matters greatly by selecting the identity transfer $g_{y \succeq y'} = 1_\G$.

\subsection{Constructing $\bA$ and $\Phi$}\label{ssec:APhi} Given a choice of lifts and transfers corresponding to simplices and face relations in $\Y$, one may construct a complex of groups $\bA:\Fac(\Y) \to \Grp$ as follows: 
\begin{enumerate}
\item for each simplex $y$, the group $\bA_y$ is the stabilizer $\G_x$, where $x = x_y$ is chosen the lift of $y$,
\item for each pair $y \succeq y'$, the map $\bA_{y \succeq y'}:\bA_y \inj \bA_{y'}$ is $\Ad(g)$, where $g = g_{y \succeq y'}$ is the transfer selected for $y \succeq y'$, and
\item for each triple $y \succeq y' \succeq y''$, the map $\bA_{y \succeq y' \succeq y''}:\bA_{y' \succeq y''} \com \bA_{y \succeq y'} \two \bA_{y \succeq y''}$ is given by the following product of transfers in $\G$:
\[
 g_{y' \succeq y''} \cd g_{y \succeq y'} \cd g\inv_{y \succeq y''}.
\]
\end{enumerate}
To confirm that the above data prescribe a bona fide complex of groups, one must check that conjugation by $g_{y \succeq y'}$ maps $\G_x$ injectively to a subgroup of $\G_{x'}$ (it does), and that the 2-morphisms satisfy the cocycle condition from Definition \ref{def:cogs} (they do)\footnote{After performing straightforward manipulations, both sides of the cocycle equation evaluate to the pleasant expression  $g_{y'' \succ y'''} \cd g_{y' \succeq y''} \cd g_{y \succeq y'} \cd g\inv_{y \succeq y'''}$.}.

\begin{rem}{\em If $\G$ is abelian, then its conjugate subgroups are necessarily equal; in this case, all the 1-morphisms $\bA_{y \succeq y'}$ are just inclusions.}
\end{rem}

Turning to the matter of constructing an injective morphism $\Phi:\bA \to \cgog^\Y$, we assign
\begin{enumerate}
\item to each simplex $y$, the inclusion $\Phi_y: \bA_y \inj \G$ (noting that $\bA_y$ is the stabilizer of $y$'s lift, and hence naturally a subgroup of $\G$), and
\item to each face relation $y \succeq y'$, the 2-morphism $\Phi_{y \succeq y'}$ given by the transfer $g_{y \succeq y'} \in \G$.
\end{enumerate}
 We leave it to the reader to confirm that these assignments successfully produce an injective morphism $\Phi:\bA \to \cgog^\Y$ as described in Remark \ref{rem:injconst}.

\subsection{Context} In this section we quantify the impact of lifts and transfers on the construction of $(\bA,\Phi)$. Let $\CG(\Y)$ be the 2-category from Definition \ref{def:2catcog}. The following {\em fiber 2-category} \cite[Sec 3]{cegarra:11} forms a natural home for pairs like the $(\bA,\Phi)$ constructed above.

\begin{defn}
\label{def:fibcat}
{\em
By $\CG(\Y) \fiber \cgog^\Y$ we denote the { fiber 2-category} whose
\begin{enumerate}
\item objects are all pairs $(\bB,\Psi)$, where $\bB$ is an object of $\CG(\Y)$ while $\Psi:\bB \to \cgog^\Y$ is a morphism, 
\item 1-morphisms $(\bB,\Psi) \to (\bB',\Psi')$ are all pairs $(\Delta,\eta)$ where $\Delta:\bB \to \bB'$ is a 1-morphism in $\CG(\Y)$ while $\eta: \Psi' \com \Delta \two \Psi$ is a homotopy: 
\[
\xymatrixrowsep{0.3in}
\xymatrixcolsep{0.27in}
\xymatrix{
\bB \ar@{->}[rr]^-{\Delta}_*+<.35em>{\stackrel{\eta}{\Longleftarrow}} \ar@{->}[dr]_-{\Psi} & & \bB' \ar@{->}[dl]^-{\Psi'} \\
& \cgog^\Y & 
}
\]
\item 2-morphisms $(\Delta,\eta) \two (\Delta',\eta')$ are given by all those homotopies $\theta:\Delta' \two \Delta$ in $\CG(\Y)$ which make the following diagram commute:
\[
\xymatrixrowsep{0.3in}
\xymatrixcolsep{0.27in}
\xymatrix{
\Psi' \com \Delta \ar@{=>}[dr]_-{\eta} & & \Psi' \com \Delta' \ar@{=>}[dl]^-{\eta'} \ar@{=>}[ll]_-{\id \circ \theta} \\
& \Psi & 
}
\]
\end{enumerate}
}
\end{defn}

\begin{proposition}
Let $(\bA,\Phi)$ and $(\bA',\Phi')$ be produced by different choices of lifts and transfers. Then, the following hold in $\CG(\Y) \fiber \cgog^\Y$:
\begin{enumerate}
\item there exists an isomorphism $(\Delta,\eta):(\bA,\Phi) \to (\bA',\Phi')$ so that each $\Delta_y:\bA_y \to \bA'_y$ is given by conjugation by some element of $\G$, and
\item given any two such isomorphisms, there exists a unique homotopy $\theta:(\Delta,\eta) \two (\Delta',\eta')$.
\end{enumerate}
\end{proposition}
\begin{proof}
Let $\{x_y,g_{y \succeq y'}\}$ and $\{z_y,h_{y \succeq y'}\}$ be two choices of lifts and transfers used to construct $\bA$ and $\bA'$ respectively. For each simplex $y$ of $\Y$, the regularity of the $\G$-action on $\X$ guarantees the presence of some $k_y \in \G$ so that $k_y \cd x_y=z_y$. Given such $k_y$, define an isomorphism $\Delta:\bA \to \bA'$ in $\CG(\Y)$ as follows: the group isomorphism $\Delta_y:\bA_y \to \bA'_y$ is $\Ad(k_y)$, while the 2-morphism $\Delta_{y \succeq y'}$ is prescribed by the element
\[
k_{y'}\cd g_{y \succeq y'} \cd k\inv_{y} \cd h\inv_{y \succeq y'}
\] in the group $\bA'_{y'}$. A homotopy $\eta:\Phi' \com \Delta \two \Phi$ is now obtained by setting $\eta_y = k\inv_y$ for each simplex $y$. Given another isomorphism $(\Delta',\eta'):(\bA,\Phi) \to (\bA',\Phi')$ generated by different choices $\{\ell_y\}$ instead of $\{k_y\}$, the desired unique homotopy $\theta:(\Delta,\eta) \two (\Delta',\eta')$ is determined completely by setting $\theta_y = k_y\cd \ell\inv_y$.
\end{proof}


\section{Reconstruction} \label{sec:recon}

We assume throughout this section that $\Y$ is a fixed finite simplicial complex and that $(\bA,\Phi)$ is a distinguished object in the fiber 2-category $\CG(\Y) \fiber \cgog^\Y$ from Definition \ref{def:fibcat}. Our goal here is to construct a simplicial complex $\X$ and a regular action of $\G$ on $\X$ that satisfies three natural criteria:
\begin{enumerate}
\item the quotient $\Quot$ is isomorphic to $\Y$, 
\item compressing this $\G$-action produces $(\bA,\Phi)$, and
\item $\X$ is unique up to $\G$-equivariant simplicial isomorphism.
\end{enumerate}
The last criterion above is a universal property: it asserts that any other $\X'$ satisfying the first two properties admits an invertible simplicial map $\bs:\X \to \X'$ so that 
\[
\bs(g \cd x) = g \cd \bs(x)
\] for each simplex $x$ in $\X$ and group element $g$ in $\G$. 

The process of unfolding $(\bA,\Phi)$ in order to extract both $\X$ and the concomitant $\G$-action is of central importance in the theory of complexes of groups, to the extent that it has been called {\em the Basic Construction} in \cite[Thm III.C.2.13]{bridson:haefliger:99} and also in \cite[Ch 5]{davis:08}.

\subsection{The Basic Construction}\label{ssec:basic}

For each simplex $y$ of $\Y$, let $\cst_y: \G \to \G/\Phi_y(\bA_y)$ be the canonical surjective map (of sets) that sends each group element $g$ to the corresponding left coset $g\cd \Phi_y(\bA_y)$. Consider the set $P$ of pairs
\[
P = \{(y,\cst_y(g)) \mid y \text{ is a simplex of } \Y \text{ and } g \in \G\},
\]
equipped with a binary relation $\btrt$,  defined as follows:
\[
(y,\cst_y(g)) \btrt (y',\cst_{y'}(g')) \quad \text{if} \quad y \succeq y' \text{ and } \cst_{y'}(g') = \cst_{y'}(g\cd \Phi\inv_{y \succeq y'}).
\]

\begin{proposition}
The relation $\btrt$ forms a well-defined partial order on $P$.
\end{proposition}
\begin{proof}
To see that $\btrt$ is well-defined on $P$, note by Remark \ref{rem:injconst} that for any $k$ in $\bA_y$, we have an equality between 
$\Ad(\Phi_{y \succeq y'}) \com \Phi_y(k)$ and $\Phi_{y'} \com \bA_{y \succeq y'}(k)$. Recalling that $\bA_{y \succeq y'}$ takes values in the group $\bA_{y'}$, it follows that 
\[
\cst_{y'}(\Phi\inv_{y \succeq y'}) = \cst_{y'}(\Phi_y(k) \cd \Phi\inv_{y \succeq y'}).
\] Therefore, for any $g$ in $\G$, we have an equality of cosets \[ \cst_{y'}(g\cdot\Phi_{y \succ y'}\inv) = \cst_{y'}(g\cdot\Phi_y(k)\cdot\Phi_{y \succ y'}\inv),\] and so $\btrt$ does not depend on our choice of $k$. Since both the reflexivity and antisymmetry of $\btrt$ are straightforward, it remains to establish its transitivity. To this end, assume we have a string of two $\btrt$-relations:
\[
(y_0,\cst_{y_0}(g_0)) \btrt (y_1,\cst_{y_1}(g_1)) \btrt (y_2,\cst_{y_2}(g_2)). 
\] By definition of $\btrt$, we obtain $y_0 \succeq y_1$ and $y_1 \succeq y_2$ in $\Y$, whence $y_0 \succeq y_2$. Moreover, using the second and first $\btrt$-relation above (in that order), we have
\begin{align*}
\cst_{y_2}(g_2) &= \cst_{y_2}(g_1 \cd \Phi\inv_{y_1 \succeq y_2}), \\
				&= \cst_{y_2}(g_0 \cd \Phi\inv_{y_0 \succeq y_1} \cd \Phi_{y_1}(h) \cd \Phi\inv_{y_1 \succeq y_2}),
\end{align*}
for some $h \in \bA_{y_1}$. By the well-definedness of $\btrt$, the coset above does not depend on $h$ and we can simplify it to
\[
\cst_{y_2}(g_2) = \cst_{y_2}(g_0 \cd \Phi\inv_{y_0 \succeq y_1} \cd \Phi\inv_{y_1 \succeq y_2}).
\]
By (\ref{eq:injcoh}) we have  
\[
\Phi\inv_{y_0 \succeq y_1} \cd \Phi\inv_{y_1 \succeq y_2} = \Phi\inv_{y_0 \succeq y_2} \mod \Phi_{y_2}(\bA_{y_2}),
\] so we obtain the desired relation $(y_0,\cst_{y_0}(g_0)) \btrt (y_2,\cst_{y_2}(g_2))$. 
\end{proof}

In fact, $(P,\btrt)$ is the poset of simplices $\Fac(\X)$ associated to our desired simplicial complex $\X$. Each pair $x = (y,\cst_y(g))$ of $P$ constitutes a simplex of dimension $\dim(y)$ in $\X$, and its faces are given by all pairs of the form $x' = (y',\cst_{y'}(g\cd\Phi\inv_{y \succeq y'}))$ where $y'$ is a face of $y$ in $\Y$. We leave it to the reader to verify that the assignments $(y,\cst_y(g)) \mapsto (y,\cst_y(h\cd g))$ parametrized by group elements $h \in \G$ yield a regular group action of $\G$ on $\X$ with quotient $\Y$. The associated orbit map $\pro_\G:\X \to \Y$ is simply given by projecting onto the first factor, i.e., $\pro_\G(y,\cst_y(g)) = y$.

\subsection{Context}
To establish that the $\G$-action on $\X$ described above will produce $(\bA,\Phi)$ when compressed along the lines of Sec \ref{sec:comp}, one chooses 
\begin{enumerate} 
\item the lift $(y,\cst_y(1_\G))$ for each simplex $y$ (where $1_\G$ is the identity element of $\G$), 
\item and the transfer $\Phi_{y \succeq y'}$ for each face relation $y \succeq y'$. 
\end{enumerate}
It remains to show that this $\G$-action on $\X$ satisfies the universal property (involving uniqueness up to equivariant isomorphism) mentioned at the beginning of this section. For our purpose, it is convenient to work not directly with the action, but rather with equivalent information contained in the orbit map $\pro_\G:\X \to \Y$. By the regularity assumption from Definition \ref{def:reg}, $\pro_\G$ forms a {\em stratified cover}\footnote{By this we mean that the fiber of $\pro_\G$ over any simplex $y$ of $\Y$ forms a nonempty finite-sheeted covering space in the sense of \cite[Ch 1.1]{hatcher:02}.} of $\Y$ so that $\G$ acts transitively on each fiber $\pro_\G\inv(y)$. We will call such maps the {\bf stratified $\G$-covers of $\Y$}, and note that they correspond bijectively with regular simplicial group actions whose orbit space is $\Y$. 

\begin{defn}
{\em
The category of stratified $\G$-covers of $\Y$, denoted $\EC(\Y)$, is defined by the following data:
\begin{enumerate}
\item its objects are all pairs $(\Z,\bq)$ consisting of finite simplicial complexes $\Z$ and stratified $\G$-covers $\bq:\Z \to \Y$,
\item the morphisms $(\Z,\bq) \to (\Z',\bq')$ are all simplicial maps $\bs: \Z \to \Z'$ which make the evident triangle\footnote{If we regard $\G$ as acting trivially on $\Y$, then this triangle commutes $\G$-equivariantly in the category of simplicial complexes.} commute:
\[
\xymatrixrowsep{0.3in}
\xymatrixcolsep{0.27in}
\xymatrix{
\Z \ar@{->}[dr]_-{\bq} \ar@{->}[rr]^{\bs} & & \Z' \ar@{->}[dl]^-{\bq'}  \\
& \Y & 
}
\]
Morphisms are composed in the usual manner for simplicial maps.
\end{enumerate}
}
\end{defn}
The following result establishes the desired universal property of the basic construction.

\begin{proposition}
\label{prop:bascon}
Any stratified $\G$-cover $\bq:\Z \to \Y$ for which the associated $\G$-action on $\Z$ produces the compressed data $(\bA,\Phi)$ is isomorphic to $\pro_\G:\X \to \Y$ in $\EC(\Y)$. 
\end{proposition}
\begin{proof}
Given a simplex $(y,\cst_y(g))$ of $\X$, let $z = z_y$ be the lift of $y$ used to generate $(\bA,\Phi)$ as described in Sec \ref{sec:comp}, so in particular $\bq(z) = y$ and the group $\bA_y$ equals the stabilizer $\G_z$. Let $\bs:\X \to \Z$ be the assignment $(y,\cst_y(g)) \mapsto g\cd z$, which we claim is the desired isomorphism. There are several properties to check, but all follow from routine calculations.
\begin{enumerate}

\item To see that $\bs$ is {\em well-defined}, note that if $\cst_y(g) = \cst_y(g')$, then $g = g'\cd k$ for some $k$ in the stabilizer $\G_z$ (recall that $\Phi_y$ is the inclusion $\G_z \inj \G$). Now,
$g \cd z = g' \cd k \cd z = g' \cd z$, so the image of $(y,\cst_y(g))$ depends only on the coset and not on $g$. 

\item Similarly, if two simplices $(y,\cst_y(g))$ and $(y',\cst_{y'}(h))$ of $\X$ are sent by $\bs$ to the same simplex $z'$ of $\Z$, then we must have $y = y' = \bq(z')$ and $z' = g\cd z = h \cd z$, which implies $\cst_y(g) = \cst_y(h)$ and hence that $\bs$ is {\em injective}.

\item Given an arbitrary simplex $z'$ in $\Z$, let $y = \bq(z)$ have lift $z$. Since $\bq$ is a stratified $\G$-cover, $\G$ acts transitively on $\bq\inv(y)$ and so there is some $h$ in $\G$ for which $h \cd z = z'$. By definition, we have $\bs(y,\cst_y(h)) = h\cd z = z'$, so $\bs$ is {\em surjective}.

\item To see that $\bs$ is {\em $\G$-equivariant}, pick $h$ in $\G$ and note that 
\[
h\cd (y,\cst_y(g)) = (y,\cst_y(h\cd g)),
\] whose image under $\bs$ is the desired $h \cd g \cd z = h \cd \bs(y,\cst_y(g))$.

\item To see that $\bs$ is a {\em simplicial} map, consider a relation $(y,\cst_y(g)) \btrt (y',\cst_{y'}(g'))$ in $\X$, so we have $y \succeq y'$ in $\Y$ and $g\cd \Phi\inv_{y \succeq y'} = g'\cd \Phi_{y'}(k)$ for some $k$ in $\bA_{y'}$. But since $\Phi_{y'}$ is just the inclusion $\bA_{y'}\inj\G$, we have $\Phi_{y'}(k) = k$. Let $z$ and $z'$ denote the lifts of $y$ and $y'$, so $\bA_{y'}$ is the stabilizer $\G_{z'}$ (whence $k$ fixes $z'$). By definition, the transfer $\Phi_{y \succeq y'}$ satisfies $z \succeq \Phi_{y\succeq y'}\inv \cd z'$. Now,
\[
z \succeq \Phi_{y\succeq y'}\inv \cd z' = g\inv \cd g' \cd k \cd z' = g\inv \cd g' \cd z',
\] 
so in fact $\bs(y,\cst_y(g)) = g \cd z$ admits $\bs(y',\cst_{y'}(g')) = g' \cd z'$ as a face.

\end{enumerate}
And finally, since $\bs:\X \to \Z$ preserves fibers, (i.e., maps $\pro_\G\inv(y)$ to $\bq\inv(y)$ for each simplex $y$ of $\Y$), we also have $\bq \com \bs = \pro_\G$ as simplicial maps. Thus, $\bs$ is the desired isomorphism from $\pro_\G$ to $\bq$ in $\EC(\Y)$. 
\end{proof}

\section{Algorithms} \label{sec:algos}

The efficiency (and often, even the feasibility) of group-theoretic algorithms varies enormously with the data structures used to store groups on a computer\footnote{Finite groups may be stored on computers as multiplication tables, lists of permutations, lists of invertible matrices coming from a representation, or sets of generators and relations. The standard reference \cite{holt:eick:obrien:05} devotes its entire third chapter to such considerations.}. Given this dependence, we will describe the requisite group-theoretic subroutines at a high level, and hope that their prospective implementer will be able to tailor data structures to specific choices of the acting group. Na\"ive implementations which work for all finite groups are straightforward to specify, but they may be quite far from optimal in practice.

For similar reasons, we do not fix a particular data structure for representing the $\G$-action. There are several reasonable options, each with its own relative (dis)advantages. For instance, one could employ a hash table with key/value assignments $(g,v) \mapsto g \cd v$ ranging over group elements $g$ in $\G$ and vertices $v$ of $\X$. 

\begin{rem}{\em Even with extremely na\"ive implementations, {\em both algorithms described below can be made to run in quadratic time} in terms of the order of $\G$ and the dimension of $\X$. The bounds are recorded in Corollary \ref{cor:polytime} below.}
\end{rem}

\subsection{Subroutines}\label{ssec:subs}

We will assume the ability to evaluate results of the two standard group operations in $\G$:
\begin{enumerate}
\item the function {\tt prod} accepts $(g, h)\in \G \times \G$ and returns their product $g \cd h$, and
\item the function {\tt inv} accepts $g \in \G$ and returns its inverse $g\inv$.
\end{enumerate}
Even for these simple operations, the data structure which holds $\G$ plays an essential role in determining the computational complexity. For instance, if $\G$ has been stored via its multiplication table, then {\tt prod} incurs a constant cost; but if $\G$ is stored as a list of permutations, then this cost may be as large as $\Oh(|\G|)$. We will denote the complexity of a single product computation in $\G$ by $r$ and a single inversion by $i$.

We also require a less standard function {\tt minrep}. Fix, once and for all, an enumeration of the elements of $\G$
\[
\iota:\G \to \{1,2,\ldots,|\G|\}, 
\] where $|\G|$ is the cardinality of $\G$. It will be convenient to assume that the identity has minimal index (i.e., $\iota(1_\G) = 1$). With this preamble in place, {\tt minrep} accepts as input a subgroup $\HH \leq \G$ along with an element $g \in \G$, and returns the $\iota$-minimal element contained in the corresponding left-coset $g\cd \HH$. In particular, if $g \in \HH$, then {\tt minrep}$(\HH,g)$ returns the identity $1_\G$. A na\"ive implementation of {\tt minrep} would involve checking the index of each group element obtained by multiplying $g$ with each element of $\HH$ --- this incurs a cost of $\Oh(|\HH|p)$. We will denote by $m$ the worst-case complexity of executing {\tt minrep} (across all choices of input $\HH$ and $g$). Calling {\tt minrep} on a fixed $\HH$ with all choices of $g \in \G$ and removing duplicates from the resulting outputs produces a {\em coset transversal} for $\G/\HH$ in the language of \cite[Ch 4.6.7]{holt:eick:obrien:05}.

The three subroutines above involved only the group $\G$; these next three also require knowledge of how $\G$ acts on $\X$:
\begin{enumerate}
	\item {\tt orb} takes as input a simplex $x$ of $\X$ and returns the list of all simplices in the orbit $\G x$, which is always a subset of $\X$;
	\item  {\tt stab} takes as input a simplex $x$ of $\X$ and returns its stabilizer $\G_x$, which is always a subgroup of $\G$; and,
	\item {\tt trans} takes as input two simplices $x$ and $x'$ of $\X$, and returns a group element $g$, if one exists, satisfying $g \cd x = x'$.
\end{enumerate}
We denote the worst-case complexity of these three algorithms by $o, s$ and $t$ respectively. Again, these complexities will vary with how $\G$ and its action on $\X$ are stored on the computer.

\subsection{The Compression Algorithm}\label{ssec:compalg}

Our first main algorithm {\tt Compress} takes as input the regular action of a finite group $\G$ on a finite simplicial complex $\X$. We will write $\Sub(\G)$ to denote the set of all the subgroups of $\G$. 

\begin{rem}
\label{rem:compout}	
{\em
The output of {\tt Compress}$(\X,\G)$ is a triple $(\Y,S,T)$, where
\begin{enumerate}
	\item  $\Y$ is the quotient simplicial complex $\X/\G$; for each $k \geq 1$, define
	\[
	\Y[k] = \{y_1 \succeq y_2 \succeq \cdots \succeq y_k \mid y_i \in \Fac(\Y) \text{ and } \dim y_i - \dim y_j = j - i\},
	\] 	
	\item $S:\Y[1] \to \Sub(\G)$ is a map that sends each simplex $y$ of $\Y$ to a subgroup $S(y) \leq \G$, which is the stabilizer of a chosen lift of $y$, and
	\item $T: \Y[2] \to \G$ is a map that sends each codimension one face relation $y_1 \succeq y_2$ in $\Y$ to a transfer element $T(y_1 \succeq y_2)$ in $\G$.
\end{enumerate}
}
\end{rem}

As the algorithm executes, it visits all the simplices of $\X$ in ascending order with respect to dimension. In order to guarantee its termination, we initially mark all simplices as unvisited (via boolean variables, for instance) and allow the algorithm to mark simplices as visited once it has processed them. {\tt Compress} also constructs two natural functions relating the input simplicial complex $\X$ and the output simplicial complex $\Y$:
\begin{enumerate}
	\item $p:\X[1] \to \Y[1]$ is the (simplicial) orbit map from Sec \ref{ssec:simpact}, while
	\item $\ell:\Y[1] \to \X[1]$ is the (not necessarily simplicial) assignment of lifts as in Sec \ref{ssec:liftrans}.
\end{enumerate} 
Thus, the composite $p\circ \ell$ is always the identity map on $\Y$, while $\ell \circ p$ sends each simplex of $\X$ to the lift chosen for its orbit class in $\Y$.

\begin{table}[h!]
\centering
{\bf Algorithm: }{\tt {Compress}} \label{alg:comp} \\
{\bf Input: }Regular $\G$-action on $\X$ \\
{\bf Output: }Triple $(\Y,S,T)$\\
{
\begin{tabular}{c|l}
\hline
01 & {\bf for each} $d$ in $(0,1,\ldots,\dim \X)$ \\
02 & \tab {\bf for each} unvisited $d$-simplex $x$ in $\X$ \\
03 & \tab \tab {\bf add} a new $d$-simplex $y$ to $\Y$ \\
04 & \tab \tab {\bf set} $\ell(y) = x$ and $S(y) = ${\tt~ stab}$(x)$ \\
05 & \tab \tab {\bf for each} $x'$ in {\tt orb}$(x)$ \\
06 & \tab \tab \tab {\bf mark} $x'$ as visited \\
07 & \tab \tab \tab {\bf set} $p(x') = y$ \\
08 & \tab \tab {\bf for each} face $z \prec x$ in $\X$ of dim $d-1$ \\
09 & \tab \tab \tab {\bf set} $p(z)$ as a face of $y$ in $\Y$ \\
10 & \tab \tab \tab {\bf set} $T(y \succeq p(z)) = \text{\tt trans}(z,\ell(p(z)))$ \\
11 & \tab \tab {\bf if} $\X$ has no more unvisited simplices \\
12 & \tab \tab\tab {\bf return} $(\Y,S,T)$ \\
\hline
\end{tabular}
}
\end{table}
To confirm that the algorithm terminates, we note that line 06 marks every simplex in the orbit of a hitherto-unvisited simplex $x$ as visited, and this $x$ must lie in its own orbit. Thus, every simplex in $\X$ is eventually visited, at which point the {\bf if} conditional in line 11 evaluates to true.

\begin{rem} \label{rem:comp} 
{\em Before moving on to the reconstruction algorithm, we highlight two features of {\bf compress} which significantly impact its computational complexity on distributed systems.
\begin{enumerate}

\item  Line 09 requires knowledge of $p$-images of lower-dimensional faces of $x$, which means that the outer {\bf for} loop in line 01 can not be parallelized --- simplices of $\X$ must be processed in an order monotone with respect to their dimension.

\item On the other hand, all simplices of a fixed dimension $d$ can be processed in parallel provided that the $(d-1)$ simplices have been already processed. In other words, the lines 02-10 may be distributed across several processors without loss of correctness.

\item Although we will not use the following fact, we note that the inner {\bf for} loops of lines 05-07 and 08-10 may also be parallelized --- for a given simplex $x$, we are not required to process its orbit $\G x$ or its faces $z \prec x$ in some prescribed serial order.

\end{enumerate}
}
\end{rem}
\subsection{The Reconstruction Algorithm}\label{ssec:recalg}

Our second algorithm {\tt Reconstruct} implements the Basic Construction of Sec \ref{ssec:basic}. It accepts as input a triple $(\Y,S,T)$ produced by {\tt Compress} (see Remark \ref{rem:compout}) along with the group $\G$ where $S$ and $T$ take their values. It returns a simplicial complex $\Z$ along with the desired regular $\G$-action. In light of Sec \ref{sec:recon}, the $d$-simplices of $\Z$ will be stored as pairs of the form $(y,g)$, where $y$ is a $d$-simplex of $\Y$ and $g$ is the $\iota$-minimal representative of some left coset of $S(y)$ in $\G$.

In addition to the subroutines already described, {\tt Reconstruct} requires a purely combinatorial procedure, {\tt uniqsort}. This takes in an unordered list of $\G$-elements and returns a list sorted along the enumeration $\iota$, with all duplicates removed. 
Line 03 uses this subroutine to find the $\iota$-minimal representative of each left-coset of the subgroup $S(y) \leq \G$ (thus the size of $M$ is exactly $|\G|/|S(y)|$). Line 07 checks whether the group elements $g'$ and $g\cd T(y \succ y')$ lie in the same left coset of $S(y')$ in $\G$ --- since $y \succeq y'$ is enforced by the previous line, we have $(y,\cst_y(g)) \btrt (y',\cst_{y'}(g'))$ in the language of Sec \ref{sec:recon}, so we add the corresponding face relation to $\Z$ in line 08.  

\begin{rem}
{\em
Since line 06 requires knowledge of all $(d-1)$-simplices present in $\Z$ before the $d$-simplex $y$ can be processed, the outer {\bf for} loop in line 01 can not be parallelized. However, all simplices of a fixed dimension $d$ may be processed at once, so the {\bf for} loop of line 02 is easily parallelized.
}
\end{rem}

\begin{table}[h!]
	\centering
	{\bf Algorithm: }{\tt Reconstruct} \label{alg:recon} \\
	{\bf In: }$(\Y,S,T)$ and $\G$ \\
	{\bf Out: } Simplicial complex $\Z$\\
	{
		\begin{tabular}{c|l}
			\hline
			01 & {\bf for each} $d$ in $(0,1,\ldots,\dim \Y)$ \\
			02 & \tab {\bf for each} $d$-simplex $y$ in $\Y$ \\
			03 & \tab \tab M = {\tt uniqsort}$\big(\{$ {\tt minrep}$(S(y),g) \mid g \in \G\}\big)$ \\
			04 & \tab \tab {\bf for each} $g$ in $M$ \\
			05 & \tab \tab \tab {\bf add} a $d$-simplex $(y,g)$ to $\Z$ \\
			06 & \tab \tab \tab {\bf for each} simplex $(y',g')$ in $\Z$ with $(y \succ y')$ in $\Y[1]$ \\
			07 & \tab \tab \tab \tab {\bf if} $\text{\tt prod}(g',\text{\tt prod}(\text{\tt inv}(g),T(y \succ y')))$ is in $S(y')$ \\
			08 & \tab \tab \tab \tab \tab {\bf set} $(y',c')$ as a face of $(y,c)$ in $\Z$\\
			09 & {\bf return }$\Z$ \\
			\hline
		\end{tabular}
	}
\end{table}

Recovering the $\G$-action on $\Z$ is straightforward --- since its simplices are stored as pairs of the form $(y,g)$, for each $h \in \G$ we have 
\[
h\cd (y,g) = \big(y,\text{\tt minrep}(S(y),\text{\tt prod}(h,g))\big).
\]
Thus, we recover not only $\Z$ but also a $\G$-action on it. Prop \ref{prop:bascon} guarantees that if the input $(\Y,S,T)$ to {\tt Reconstruct} was produced by running {\tt Compress} on a simplicial complex $\X$ with a regular $\G$-action, then $\X$ and $\Z$ are $\G$-equivariantly isomorphic.

\subsection{Complexity Analysis}
We analyze both {\tt Compress} and {\tt Reconstruct} in terms of the complexity parameters $p,i,m,o,s$ and $t$ associated to the six subroutines of Sec \ref{ssec:subs}.

\begin{thm}
	\label{thm:comp}
The computational complexity of running {\tt Compress} for the regular action of a group $\G$ on a finite $n$-dimensional simplicial complex $\X$ is
\[
\Oh\Big( (n+1) \cd (s+o+f+t \cd (n+1))\Big), \text{ where}
\]
\begin{enumerate}
	\item $s,o$ and $t$ are the complexity parameters for {\tt stab}, {\tt orb} and {\tt trans}, while
	\item $f$ is the maximal orbit-length\footnote{Note that this $f$ is also the maximal stabilizer index $\max_{x \in \X} 	\{[\G:\G_x]\}$ and the maximal fiber cardinality $\max_{y \in \X/G} \{|\pro_G\inv(y)|\}$ of the orbit map.} encountered among the simplices of $\X$,
	\[ 
	f = \max_{x \in \X}\{ |\G x| \},
	\]
\end{enumerate} 
provided that the number of available processors exceeds the number of $d$-simplices in $\X$ for each dimension $d \in \{0,1,\ldots,n-1,n\}$.
\end{thm}
\begin{proof}
The outer {\bf for} loop in line 01 of {\tt Compress} runs exactly $(n+1)$ times corresponding to $d$-values $\{0,1,\ldots,n\}$. As noted in Remark \ref{rem:comp}, the $d$-simplices $x$ of $\X$ can be processed independently of each other for each fixed $d$. So by our assumption on the number of processors, it suffices to only measure the complexity of executing lines 03-10 once. Now line 04 incurs the cost of running {\tt stab} once, while the {\bf for} loop in line 05-07 runs {\tt orb} once and iterates at most $f$ times. Finally, the {\bf for} loop in lines 08-09 runs once for each codimension one face of $x$ and calls {\tt trans} each time; there are at most $(n+1)$ such faces since $\dim x \leq \dim \X = n$. Combining these contributions, the cost of processing a single $d$-simplex $x$ is \[\Oh(s+o+f+t\cd (n+1)).\] The desired complexity estimate now follows from the fact that the outermost loop executes $(n+1)$ times.
\end{proof}

Based on the preceding estimate, we expect that the cost of running {\tt Compress} will be dominated by the calls to {\tt trans} in line 10. It is therefore of compelling interest to optimize the implementation of {\tt trans} to the largest extent possible. We recommend trying to find a large subset $\mathbb{T} \subset \X$ to populate many lifts $\ell(y)$ via a $\G$-equivariant version of breadth-first search on $\X$ (when a simplex is visited, mark every simplex in its orbit as visited). The larger this $\mathbb{T}$, the more frequently we will have equalities $z = \ell(p(z))$ in line 10 of {\tt Compress}. When its two inputs are equal, {\tt trans} is allowed to simply return the identity $1_G$ and hence incur a constant cost.

Finally, we turn to {\tt Reconstruct}.
\begin{thm}
	\label{thm:recon}
The computational complexity of running {\tt Reconstruct} on the input $(\Y,S,T;\G)$ when $\Y$ is a finite $n$-dimensional simplicial complex is
\[
\Oh\Big((n+1)\cd\big[k \cd(m + \log_2 k) + f \cd (n+1) \cd (2r+i+h)\big]\Big), \text{ where}
\]
\begin{enumerate}
\item $k = |\G|$ is the order of $\G$,
\item $m, r$ and $i$ are the complexity parameters for {\tt minrep}, {\tt prod} and {\tt inv}, while
\item $f = \max_{y \in \Y}\{k/|S(y)|\}$ and $h = \max_{y \in \Y}\{|S(y)|\}$ (both are $\leq k$),
\end{enumerate}
provided that the number of available processors exceeds the number of $d$-dimensional simplices in $\Y$ for each $d$ in $\{0,1,\ldots,n-1,n\}$.
\end{thm}
\begin{proof}
We examine the nested {\bf for} loops from the inside out, starting with the innermost loop of lines 06-08. This loop runs once for each $(d-1)$-dimensional face $y'$ of $y$, and since $\dim \Y= n$ there are at most $(n+1)$ such faces. For each such face, {\tt prod} is invoked at most twice and {\tt inv} at most once in line 07. We then check whether an element of $\G$ lies in $S(y')$, a list of size bounded by $h$. Thus, our innermost for loop incurs a $\Oh((n+1) \cd (2r+i+h))$ cost per iteration. The intermediate {\bf for} loop spanning lines 04-08 runs at most $f$ times since $f$ bounds from above the index $[\G:S(y)]$ and hence the size of $M$. Thus, the cost of running this intermediate loop is $\Oh(f\cd (n+1)\cd(2r+i+h))$. By our assumption on the number of processors, it suffices to run the {\bf for} loop of lines 02-08 only once per dimension $d$, and it remains to account for line 03. Here we first call {\tt minrep} exactly $k$ times (cost $\Oh(k\cd m$)), and then sort the resulting list of $k$ outputs and remove duplicates (cost $\Oh(k \log_2 k)$). Thus, when processing simplices of a fixed dimension $d$, {\tt Reconstruct} incurs a computational cost of
\[
\Oh(k\cd(m+\log_2 k)+f\cd (n+1)\cd (2r+i+h)).
\]
Since the outermost {\bf for} loop of lines 01-08 runs exactly $(n+1)$ times, we obtain the desired estimate.	
\end{proof}

By over-estimating most of the auxiliary complexity parameters described above in terms of the order of $\G$, it becomes evident that both {\tt compress} and {\tt reconstruct} will run in polynomial time. As before, let $k$ be the order of the acting group $\G$. We can estimate the complexity parameters used in Theorems \ref{thm:comp} and \ref{thm:recon} as follows.
\begin{enumerate}
	\item Both $r$ and $i$ are bounded by $k$, since in the worst case we have to search every element of $\G$ to find a product or an inverse. Similarly $m$ is bounded above by $k$ since in the worst case one may search all of $\G$ to find the minimal representative of a given coset.
	\item The parameter $o$ is bounded by $k$ since one only has to compute the action of every group element on a given simplex. For the same reason, the orbit-length parameters $f$ (from both theorems) and $h$ (from Theorem \ref{thm:recon}) are bounded by $k$.
	\item Meanwhile, both $s$ and $t$ are bounded by $k$ since in the worst case one has to test every element of $\G$ to check whether it stabilizes a simplex or sends one simplex to another.
\end{enumerate}
The following result follows directly from using these over-estimates in the complexities from Theorems \ref{thm:comp} and \ref{thm:recon} along with the elementary properties of $\Oh(\cdot)$.

\begin{corollary}\label{cor:polytime}
Let $\G$ be a group of order $k$ that acts regularly on a finite simplicial complex $\X$ of dimension $n$. Assume that the number of available processors exceeds the number of simplices in $\X$. Then the computational complexity of running {\tt compress} for this action is no worse than
\[
\Oh\Big(k \cdot (n+1)^2\Big).
\]	
And the complexity of running {\tt reconstruct} on the output of {\tt compress} is no worse than 
\[
\Oh\Big(k^2 \cdot (n+1)^2\Big).
\]
\end{corollary}

\section*{Acknowledgements}
We are grateful to Scott Murray for providing us with a Magma \cite{magma} implementation of both algorithms, as well as for his many useful suggestions which improved this paper. LC is funded by the Simons Foundation Collaboration grant number 422182. VN is grateful to John Meier for timely group-theoretic advice and to Martin Bridson for the loaned copy of \cite{bridson:haefliger:99}. His work was supported by the Friends of the Institute for Advanced Study.

\bibliographystyle{abbrv}
\bibliography{symbib}

\end{document}